\documentclass[12pt]{amsart}%
\usepackage{amsmath}%
\usepackage{amssymb}%
\usepackage{amscd}%
\usepackage{color}%
\usepackage{enumitem}%
\usepackage{array}
\usepackage{graphicx}%
\usepackage{mathrsfs}%
\usepackage[all]{xy}%
\usepackage{stmaryrd}
\usepackage{newtxtext,newtxmath}%
\usepackage[all, warning]{onlyamsmath}
\usepackage{amsrefs}
\usepackage{bbm}
\usepackage{trimclip}
\usepackage[capitalise]{cleveref}
\usepackage[displaymath, mathlines]{lineno}
\makeatletter
\DeclareRobustCommand{\arr}{%
 \mathrel{\mathpalette\short@to\relax}%
}
\newcommand{\short@to}[2]{%
  \mkern2mu
  \clipbox{{.3\width} 0 0 0}{$\m@th#1\vphantom{+}{\shortrightarrow}$}%
  }
\makeatother
\theoremstyle{plain}
    
    \newtheorem{theorem}{Theorem}[section]
    \newtheorem{proposition}[theorem]{Proposition}
    \newtheorem{lemma}[theorem]{Lemma}
    \newtheorem{corollary}[theorem]{Corollary}

\theoremstyle{definition}
    \newtheorem{definition}[theorem]{Definition}
    
    \newtheorem{example}[theorem]{Example}
    \newtheorem{remark}[theorem]{Remark}
\def\Alphabet{A,B,C,D,E,F,G,H,I,J,K,L,M,N,O,P,Q,R,S,T,U,V,W,X,Y,Z}%  Capitalized Alphabet
\def\grabet{a,b,c,d,e,f,g,h,i,j,k,l,m,n,o,p,q,r,s,t,u,v,w,x,y,z}%	lowercase alphabet
\def\endpiece{xxx}%									marks end of list
\def\makeAlphabet[#1]{\expandafter\makeA#1,xxx,}%		Xx. \makeAlphabet[A,B]
\def\makealphabet[#1]{\expandafter\makea#1,xxx,}%		Xx. \makealphabet[c,d]
\def\makeA#1,{\def\temp{#1}\ifx\temp\endpiece\else%
\mkbb{#1}\mkfrak{#1}\mkbf{#1}\mkcal{#1}\mkscr{#1}\mkbs{#1}\expandafter\makeA\fi}%
\def\makea#1,{\def\temp{#1}\ifx\temp\endpiece\else\mkfrak{#1}\mkbf{#1}\mkbs{#1}\expandafter\makea\fi}%
\def\mkbb#1{\expandafter\def\csname bb#1\endcsname{\mathbb{#1}}}%      Define bb
\def\mkfrak#1{\expandafter\def\csname fr#1\endcsname{\mathfrak{#1}}}%    Define frak
\def\mkbf#1{\expandafter\def\csname b#1\endcsname{\mathbf{#1}}}%           Define bold letters
\def\mkcal#1{\expandafter\def\csname c#1\endcsname{\mathcal{#1}}}%       Define calligraphy
\def\mkscr#1{\expandafter\def\csname s#1\endcsname{\mathscr{#1}}}%       Define script
\def\mkbs#1{\expandafter\def\csname bs#1\endcsname{{\boldsymbol{#1}}}}%       Define bold symbol
\def\makeop[#1]{\xmakeop#1,xxx,}%					Xx. \makeop[Hom,Spec]
\def\mkop#1{\expandafter\def\csname #1\endcsname{{\mathrm{#1}}}} % 
\def\xmakeop#1,{\def\temp{#1}\ifx\temp\endpiece\else\mkop{#1}\expandafter\xmakeop\fi}%
\def\makeup[#1]{\xmakeup#1,xxx,}%					Xx. \makeop[Hom,Spec]
\def\mkup#1{\expandafter\def\csname #1\endcsname{{\mathrm{#1}\,}}} % 
\def\xmakeup#1,{\def\temp{#1}\ifx\temp\endpiece\else\mkup{#1}\expandafter\xmakeup\fi}%
\makeAlphabet[\Alphabet]%				Define bb, frak, bf, cal for Capitalized Alphabet
\makealphabet[\grabet]%  				Define frak and bf for uncapitalized alphabet
\makeop[Map,Hom,alt,id,loc,col,unif,len,Supp,Consv,pr,Aut,Xxt,ab,tors,Coind,univ,sp,diam,Span,ord]% 		Homs
\makeup[Im,Ker,Coker,rank]
\def\e{\eta}

\def\La{\Lambda}
\def\Z{\bbZ}
\def\E{\bbE}

\def\R{\bbR}

\def\abs#1{|#1|}

\def\Dabs#1{\bigl|\kern-0.3mm\bigl|#1\bigr|\kern-0.3mm\bigr|}
\def\C{\operatorname{Consv}^\phi}

\def\sd{{*'}}
\def\sdd{{*''}}

\def\p{{\partial}}

\def\ms{{\operatorname{ms}}}

\def\EX{{\operatorname{ex}}}

\begin{document}

\setcounter{tocdepth}{2}

%--- the title ------------------------------------------------------------------------------------------------------------------------
\title[Uniform Functions]{On Uniform Functions on Configuration Spaces of Large Scale Interacting Systems}
\author[Bannai]{Kenichi Bannai}\email{bannai@math.keio.ac.jp}
\author[Sasada]{Makiko Sasada}\email{sasada@ms.u-tokyo.ac.jp}
\thanks{This work is supported by JST CREST Grant Number JPMJCR1913 and KAKENHI 24K21515.}
\address[Bannai]{Department of Mathematics, Faculty of Science and Technology, Keio University, 3-14-1 Hiyoshi, Kouhoku-ku, Yokohama 223-8522, Japan.}
\address[Sasada]{Department of Mathematics, University of Tokyo, 3-8-1 Komaba, Meguro-ku, Tokyo 153-0041, Japan.}
\address[Bannai, Sasada]{Mathematical Science Team, RIKEN Center for Advanced Intelligence Project (AIP),1-4-1 Nihonbashi, Chuo-ku, Tokyo 103-0027, Japan.}

\date{\today}
\begin{abstract}
	Stochastic large scale interacting systems
	can be studied via the observables, i.e.\ functions on the underlying configuration space.
	In our previous article, we introduced the concept of \emph{uniform functions}, which are
	suitable class of functions on configuration spaces underlying stochastic systems
	on infinite graphs. An important consequence is the successful characterization of 
	conserved quantities without introducing the notion of stationary distributions.
	In this article, we further 
	develop the theory of
	uniform functions and construct the theory independent of any choice of a base state.
	Furthermore, we generalize the notion of interactions given in our previous article	
	to accommodate the case where
	there are multiple possible state transitions on adjacent vertices.
	We then prove that
	 if the interaction is \emph{exchangeable}, then
	 any uniform function which gives a global conserved quantity 
	 can be expressed as a sum of local conserved quantities of the interaction.
	Contrary to our previous article,
	 we do not need to assume that
	the interaction
	 is irreducibly quantified.
	 This shows that our theory of uniform functions on
	 configuration spaces over infinite graphs
	 with transition structure given by an exchangeable interaction
	 is a natural framework to study general stochastic large scale interacting systems. 
	 While some of the ideas in this article are based on our previous article, the article is logically independent and self-contained.
\end{abstract}

\subjclass[2020]{Primary: 82C22, Secondary: 05C63} 
\keywords{uniform function, configuration space, large scale interacting system, interaction, conserved quantity, cohomology}
\maketitle
%\tableofcontents

%%%%%%%%%%%%%%%%%%%%%%%%%%%%%%%%%%%%%%%%%%%%%%%%%%
%
%
%
\section{Introduction}\label{sec: introduction}
%
%
%
%%%%%%%%%%%%%%%%%%%%%%%%%%%%%%%%%%%%%%%%%%%%%%%%%%

Deriving the macroscopic evolution from the dynamics of microscopic systems is a very fundamental and challenging task. As a mathematically rigorous theory, hydrodynamic limits for stochastic large scale interacting systems have been widely studied. In order to provide a new perspective for the analysis of such large scale interacting systems, in our previous article  \cite{BKS20}, we introduced the concept of configuration space with transition 
structure and the associated space of \emph{uniform functions}.
%For more on the background, see \cite{BKS20}.
%In our previous article \cite{BKS20}, we introduced the concept of 
The uniform functions form a suitable class of functions including the conserved quantities on 
configuration spaces on infinite graphs
underlying stochastic large scale interacting systems.
The theory was extended in \cite{BS21L2} under certain assumptions to give a proof of 
Varadhan's decomposition for hydrodynamic limits.
In \cite{BKS20} and \cite{BS21L2}, the theory was constructed using a choice of a 
fixed base state of the local state space.
In this article, we construct the theory independent of any 
such choice of base state.  
Furthermore, we expand the notion
of interactions to accommodate the case when the 
transitions between states on adjacent vertices is given by a general 
symmetric \emph{digraph} (i.e.\ \emph{directed graph}),
which we again call an \emph{interaction}.
We then prove that if the interaction is \emph{exchangeable},
then any uniform function on the configuration space
invariant under transitions can be
expressed as the sum of the local conserved quantities of the interaction.
This is a generalization of \cite{BKS20}*{Theorem 3.7}.
%Our result implies that the global conserved quantities of the
%entire large scale interacting system given by a uniform function
%can be expressed as the sum of the conserved quantities of the interaction.
Although we build on ideas initiated in \cite{BKS20},
our article is logically independent of such results.

\begin{figure}[htbp]
	\begin{center}
		\includegraphics[width=0.8\linewidth]{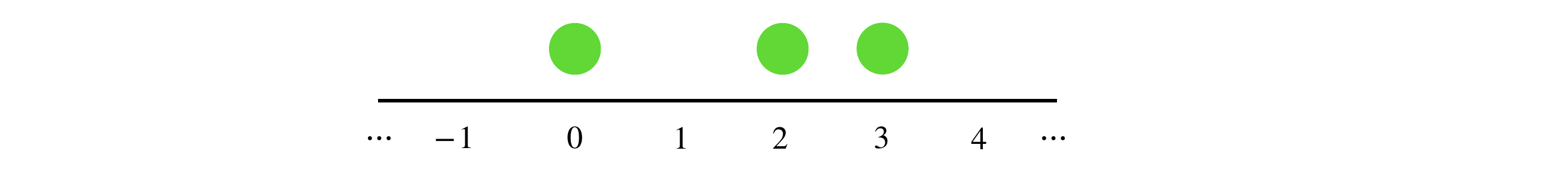}
		\caption{Example of a configuration in the Configuration Space $S^X=\{0,1\}^\Z$.}\label{Fig: 01}
	\end{center}
\end{figure}

Let $S$ be a non-empty set, which we call the \emph{local state space}.
The fundamental example is $S=\{0,1\}$, where $0$ may express that 
the site is vacant and $1$ may express that the site is occupied by a particle.
For a symmetric digraph $(X,E)$ with set of vertices $X$ and set of edges $E\subset X\times X$,
the configuration space of $S$ on $(X,E)$ is defined as $S^X\coloneqq\prod_{x\in X}S$.
We call any $\e=(\e_x)_{x\in X}\in S^X$ a \emph{configuration}.
The fundamental example of such graph is given by the \emph{one}-dimensional Euclidean lattice $(\Z,\E)$, 
where $\E=\{ (i,j)\in \Z\times\Z\mid |i-j|=1 \}$, and the configuration space
$S^X=\{0,1\}^\Z$ describes all of the possible combination of states on the Euclidean lattice.

The configurations will be quantified via the observables, i.e.\ functions $f\colon S^X\rightarrow\R$ on the
configuration space.  The premise of our model is that the observables
should depend only on the local states in the vicinity of the point of observation.  
Hence given an observable $f_x\colon S^X\rightarrow\R$ at $x\in X$,
 the value $f(\e)$ for a configuration $\e=(\e_x)_{x\in X}\in S^X$ should 
 depend on the components near $x$.
The graph structure $(X,E)$ induces the graph distance $d_X$ on the vertices in $X$,
i.e.\ the length of the shortest path between two vertices. 
We say that a function $f_x\colon S^X\rightarrow\R$ for $x\in X$ is \emph{local},
if there exists $R>0$ such that the value $f_x(\e)$ depends only on $\e_z$ 
such that $d_X(x,z)\leq R$.
In this case, we say that $f_x$ is \emph{local at $x$ with radius $R$}.
Such functions will play the role of the observables in the vicinity of $x$.
Given a system of functions $(f_x)_{x\in X}$ with $f_x$ local at $x$,
in order to get a suitable quantity for the entire system, we would like to consider the sum
\begin{equation}\label{eq: sum}
	f\coloneqq\sum_{x\in X}f_x.
\end{equation}
If the original function $f_x$ expresses the number of particles or energy or any other quantity
for the local state at $x$, then $f$ would express the total number of particles or total energy or 
the total of any other quantity for the entire system on $X$.  In considering the hydrodynamic limits
of large scale interacting systems, it is useful to consider cases when $X$ is infinite.
However, the sum \cref{eq: sum} would be an infinite sum and generally not well-defined for such $X$.
Variants of infinite sums of the form \cref{eq: sum} appear in  \cite{KL99}*{p.144}
and \cite{KLO94}*{p.1477}, but with the caveat that 
the infinite sum ``does not make sense''.  
In \cref{def: uniform}, we define uniform functions to be a certain sum of local functions on $S^X$,
which gives a rigorous definition for sums such as \cref{eq: sum}, even for the case when $X$ is infinite.

In what follows, we assume that the symmetric digraph $(X,E)$ is connected and \emph{locally finite},
i.e.\ the set $E_x\coloneqq \{ (x,y) \in E\mid y\in X \}$ is a finite set 
for any $x\in X$.  For example, the Euclidean lattice $(\Z,\E)$ is locally finite.
Our first result is the following.

\begin{proposition}[=\cref{prop: sum}]\label{prop: one}
	Let $S$ be a non-empty set, and assume that $(X,E)$ is connected and locally finite.
	Consider a system of functions $(f_x)_{x\in X}$ on $S^X$.  If the system is uniformly local, i.e.\ if
	there exists $R>0$ such that $f_x$
	is local at $x$ for radius $R$, then the sum
	\[
		f\coloneqq\sum_{x\in X}f_x
	\]
	defines a uniform function, a class of functions defined in \cref{def: uniform}.
\end{proposition}

\cref{prop: one} indicates that the infinite sum \cref{eq: sum} has meaning as a uniform function.
In fact, we will prove that any uniform function may be obtained as the sum of a system of functions 
on $S^X$ which are uniformly local.

In constructing our large scale interacting system, we express the possible transition
of states on adjacent vertices of the underlying graph.  In our case, this is given by 
an \emph{interaction} $(S,\phi)$, which we define 
$\phi$ as a subset of $(S\times S)\times (S\times S)$ such that $(S\times S,\phi)$ is a symmetric digraph.
For the case $S=\{0,1\}$ described in \cref{Fig: 02}, the configuration $(0,0)\in S\times S$
expresses the state with no particles, $(1,0)\in S\times S$ expresses the state with a particle in the first site and no particles in the second site, etc.  Then the exclusion 
$\phi_\EX\coloneqq\{((1,0),(0,1)),((0,1),(1,0))\}\subset(S\times S)\times(S\times S)$
expresses the rule that a particle may move only if the adjacent site is vacant.

\begin{figure}[htbp]
	\begin{center}
		\includegraphics[width=0.9\linewidth]{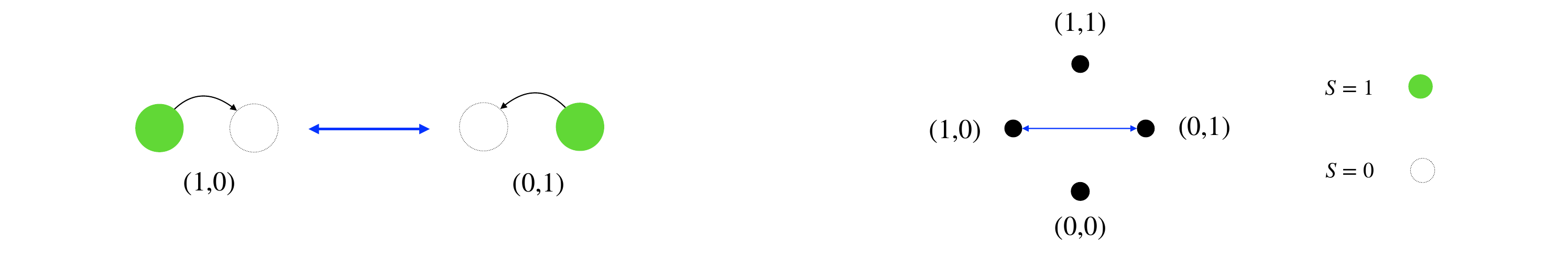}
		\caption{The exclusion $\phi_\EX$ for $S=\{0,1\}$
		expresses the rule that a particle may move only if the adjacent site is vacant.
		The diagram on the right expresses the graph $(S\times S,\phi_\EX)$.}\label{Fig: 02}
	\end{center}
\end{figure}

A choice of an interaction gives the transitions of $S^X$.  Namely, for $\e=(\e_x), \e'=(\e_x)\in S^X$,
the configuration $\e$ may transition to $\e'$ if and only if there exists an edge $e=(x,y)\in E$ such that $\e_z=\e'_z$ for $z\neq x,y$ and $((\e_x,\e_y),(\e'_x,\e'_y))\in\phi$.
If we denote by $\Phi_E\subset S^X\times S^X$ the set of all such permitted transitions, then $(S^X,\Phi_E)$ form a
symmetric digraph.
Our construction allows not only nearest neighbor models but more general models by changing the graph
$(X,E)$.  For the exclusion $\phi_\EX$, 
if we take the graph $(X,E)$ to be the Euclidean lattice 
$(\Z,\E)$, then this model underlies the nearest neighbor exclusion process.  
The exclusion process is one of the most fundamental models of the interacting particle systems 
and has been studied extensively \cites{Spo91,KL99,KLO94,KLO95,Lig99,FHU91,FUY96,GPV88,VY97}.

\begin{figure}[htbp]
	\begin{center}
		\includegraphics[width=1\linewidth]{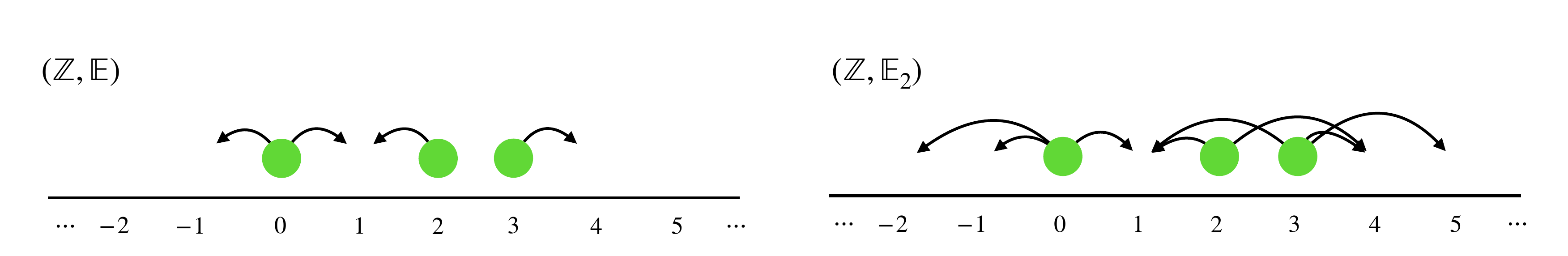}
		\caption{The transitions on $\{0,1\}^\Z$ induced by the exclusion $\phi_\EX$
		for the cases where the underlying graphs are $(\Z,\E)$ and $(\Z,\E_2)$.
		The interaction takes place between vertices connected by an edge of the underlying graph.}\label{Fig: 31}
	\end{center}
\end{figure}

Moreover, if we take the graph $(X,E)$ to
be $(\Z,\E_k)$ with $\E_k\coloneqq\{ (i,j)\in\Z\times\Z \mid 1\leq \abs{i-j}\leq k \}$ for an integer $k\geq 1$, then this model underlies the exclusion process which allows hopping to vacant sites of distance up to $k$.
This versatility of the model is one reason that 
we have separated the interaction from the underlying graph.

Given an interaction $\phi\subset (S\times S)\times (S\times S)$, we say that a function
$\xi\colon S\rightarrow\R$ is a \emph{conserved quantity} for $\phi$, if the function
$\tilde{\xi}\colon S\times S\rightarrow\R$ given by
\begin{equation}\label{eq: CQ1}
	\tilde{\xi}(s_1,s_2)\coloneqq\xi(s_1)+\xi(s_2),\qquad\forall (s_1,s_2)\in S\times S
\end{equation}
is constant on the connected components of the graph $(S\times S,\phi)$.
In other words, $\xi$ is a function whose sum is preserved by transitions of an interaction.

For a conserved quantity $\xi$, let $\xi_x(\e)\coloneqq\xi(\e_x)$ for any $\e=(\e_x)\in S^X$.
Then by definition, $\xi_x$ for any $x\in X$ is local at $x$ with radius $0$, hence by \cref{prop: one},
\[
	\xi_X\coloneqq\sum_{x\in X}\xi_x
\]
defines a uniform function.
%By \cref{eq: CQ1}, the uniform function $\xi_X$ is constant on the connected components of $(S^X,\Phi_E)$.
For the case $(\{0,1\},\phi_\EX)$, the function $\xi\colon\{0,1\}\rightarrow\R$
given by $\xi(s)=s$ is a conserved quantity of $(\{0,1\},\phi_\EX)$.
If $X$ is finite,
then $\xi_X\colon \{0,1\}^X \rightarrow\R$ is a function such that
$\xi_X(\e)=\sum_{x\in X}\xi(\e_x)$ for $\e=(\e_x)\in \{0,1\}^X$ gives the total number of particles of $\e$.

\begin{definition}\label{def: exchangeable}
	We say that an interaction $(S,\phi)$ 
	is \emph{exchangeable} if and only if 
	for any $(s_1,s_2)\in S\times S$,
	the configurations
	$(s_1,s_2)$ and $(s_2,s_1)$ are in the same connected component of 
	the graph $(S\times S,\phi)$.
\end{definition}

The interaction $(\{0,1\},\phi_\EX)$ is exchangeable.
The article \cite{BKSWY24} gives classifications and examples of various 
exchangeable interactions, hence there are an abundance of examples of exchangeable interactions.
Let $(X,E)$ be connected and locally finite 
as in \cref{prop: one}.
%For any configurations $\e,\e'\in S^X$, a transition from $\e$ to $\e'$ is a finite path from $\e$ to $\e'$ in $(S^X,\Phi_E)$.
Since a uniform function is given by an infinite sum, it is in general not an actual function on the configuration space.
However, one may define the difference of values of a uniform function
at $\e,\e'\in S^X$ with respect to any transition $(\e,\e')\in\Phi_E$. 
Hence a uniform function behaves as a potential on the configuration space.

We say that a uniform function on $S^X$ is invariant via transitions of $(S^X,\Phi_E)$,
if the difference of values with respect to any transition is \emph{zero}.  The main theorem of our article is as follows.

\begin{theorem}[=\cref{thm: main}]\label{thm: 1}
	Let $(S,\phi)$ be an interaction which is exchangeable, 
	and assume that $(X,E)$ is connected, locally finite and infinite.
	Suppose $f$ is a uniform function on $S^X$
	which is invariant via transitions of
	$(S^X,\Phi_E)$.  
	Then there exists a conserved quantity $\xi\colon S\rightarrow\R$ such that $f=\xi_X$.
\end{theorem}

A function which is invariant via transitions of $S^X$ may be
viewed as a global conserved quantity of $S^X$.
\cref{thm: 1} implies that any global conserved quantity expressed by a uniform function
is given at the sum of a local conserved quantity of the interaction.
Hence our theory of uniform functions gives a class of well-behaved
global conserved quantities of the configuration space.
Contrary to \cite{BKS20}, we do not need to assume that $(S,\phi)$ is irreducibly quantified.

\cref{thm: 1} may be interpreted in terms of uniform cohomology.
Let $C_\loc(S^X)$ be the space of local functions, and let $C^0_\loc(S^X)$ be the space of local functions modulo the subspace of constant functions.  In other words, $C^0_\loc(S^X)\coloneqq C_\loc(S^X)/\sim$, where $f\sim g$ for $f,g\in C_\loc(S^X)$ if and only if $f-g$ is a constant function.
We may define the differential 
\[
	\p\colon C^0_\loc(S^X)\rightarrow C^1(S^X)
\]
by
$
	\p f(\e,\e')\coloneqq f(\e')-f(\e)
$
for any
$
	(\e,\e')\in\Phi_E,
$
where $C^1(S^X)$ is the subspace of $\Map(\Phi_E,\R)$ consisting of functions $\omega\colon\Phi_E\rightarrow\R$
such that $\omega(\e,\e')=-\omega(\e',\e)$ for any $(\e,\e')\in\Phi_E$.
Denote by $C^0_\unif(S^X)$ the space of uniform functions modulo the subspace of constant functions.
We may  linearly extend the above differential to obtain a differential 
\[
	\p\colon C^0_\unif(S^X)\rightarrow C^1(S^X).
\]
We define the $0$-th uniform cohomology by
\[
	H^0_\unif(S^X)\coloneqq\Ker\bigl(\p\colon C^0_\unif(S^X)\rightarrow C^1(S^X)\bigr).
\]
As an analogy of the fact that $0$-th cohomology $H^0(X)$ coincides with the space of functions
constant on the connected components of the graph $(X,E)$ (see \cite{BKS20}*{Proposition A.8}),
we may prove that the $0$-th uniform cohomology $H^0_\unif(S^X)$
coincides with the space of uniform functions invariant via transitions of $(S^X,\Phi_E)$,
i.e.\
constant on the connected components of $(S^X,\Phi_E)$.
Combining \cref{thm: 1} with this fact, we obtain the following corollary.

\begin{corollary}[=\cref{cor: main}]\label{cor: intro}
	Suppose $(S,\phi)$ and $(X,E)$ are as in \cref{thm: 1}.
	Then we have a canonical isomorphism
	\[
		\C(S)\cong H^0_\unif(S^X), \qquad \xi\mapsto\xi_X,
	\]
	where $\C(S)$ denotes the space of conserved quantities of $(S,\phi)$ modulo the subspace of constant functions.
\end{corollary}

If $X$ is infinite, then $(S^X,\Phi_E)$ in general has an infinite number of connected components.
Hence the graph cohomology $H^0(S^X)$, which coincides with the space of functions constant
on the connected components of $(S^X,\Phi_E)$ is in general infinite dimensional.
However, if $\C(S)$ is finite dimensional, as in the case when $S$ is a finite set,
then \cref{cor: intro} implies that the space $H^0_\unif(S^X)$, which may be interpreted as the 
space of uniform functions constant on the connected components of $(S^X,\Phi_E)$,
is finite dimensional.
The theory of uniform functions allows for the construction of such well-behaved cohomology theory.
We emphasize that Theorem \ref{thm: 1} and Corollary \ref{cor: intro} hold in general only when $X$ is infinite. 

\medskip

The precise contents of this article are as follows.  In \S\ref{sec: UF},
we define the space of uniform functions on a configuration space 
$S^X$ of a local state space $S$ over a symmetric digraph $(X,E)$,
and prove that the space of uniform functions is independent of any choice of the base state.
We then define the notion of a uniform system of local functions and prove
in \cref{prop: sum}
that if $(X,E)$ is connected and locally finite, then class of uniform functions coincides with the class of 
functions that may be expressed as the sum of a uniform system of local functions.
In \S\ref{sec: CS}, we consider an interaction $(S,\phi)$ and the configuration space
with transition structure $(S^X,\Phi_E)$ over $(X,E)$.
Assuming that the interaction $(S,\phi)$ is exchangeable and 
the graph
$(X,E)$ is infinite, we prove our main result (\cref{thm: main}).
In \S\ref{sec: cohomology}, we introduce the differential for uniform functions and prove \cref{cor: main}.

%%%%%%%%%%%%%%%%%%%%%%%%%%%%%%%%%%%%%%%%%%%%%%%%%%
%
%
%
\section{Uniform Functions on Configuration Spaces}\label{sec: UF}
%
%
%
%%%%%%%%%%%%%%%%%%%%%%%%%%%%%%%%%%%%%%%%%%%%%%%%%%

In this section, we give the definition of uniform functions on the space of configurations on 
a symmetric digraph.  We then prove in \cref{prop: independent} that the space of uniform
functions for distinct choice of the base state $*\in S$ are canonically isomorphic.
Let $S$ be a non-empty set, which we call the \emph{local state space}.
An element of $S$ expresses a state of a system on a single site.
Furthermore, let $(X,E)$ be a symmetric digraph.  In other words, $X$ is a set called the \emph{set of vertices},
$E\subset X\times X$ is a set called the \emph{set of edges}, and we assume that $e=(oe,te)\in E$ if and only if
 $\bar e=(te,oe)\in E$.
The set $X$ gives the underlying space for the system.  The edges $E$ express the vertices in $X$ whose local states interact with each other. 
We say that the graph $(X,E)$ is \emph{locally finite}, 
if $E_x\coloneqq\{e\in E\mid oe=x\}$ is a finite set for any $x\in X$. 
Here $oe\in X$ is the origin of the edge $e$.
In what follows, we will assume that the graph $(X,E)$ is connected and locally finite.

\begin{definition}
	We define the \emph{configuration space} of states $S$ on $(X,E)$ by $S^X\coloneqq\prod_{x\in X}S$.
\end{definition}

We call any element $\e=(\e_x)\in S^X$ a \emph{configuration}.
The configuration space expresses all of the possible configuration of states that our model may take.
The configurations will be quantified via the observables, i.e.\ functions $f\colon S^X\rightarrow\R$ on the
configuration space.  
We say that $f\colon S^X\rightarrow\R$ is a \emph{local function}, if there exists a finite $\La\subset X$
such that the value $f(\e)$ depends only on the states on $\La$, i.e.\
$(\e_x)_{x \in\La}\in S^\La$.  Such function corresponds to a function in $C(S^\La)$,
where $C(S^\La)\coloneqq\Map(S^\La,\R)$ denotes the space of real valued functions on $S^\La$.
If we let $\sI$ be the set of finite subsets $\La\subset X$, then 
the space of local functions is given as $C_\loc(S^X)=\bigcup_{\La\in\sI}C(S^\La)$.

Fix $*\in S$, which we call the base state, and let $S^X_*$ be the subset of $S^X$ consisting of configurations
whose components are $*$ for all but finite $x\in X$.
We denote by $\star$ the configuration in $S^X$ whose components are
all at base state.  Then $\star\in S^X_*$.
For $\La\subset X$, we have $C(S^\La)\subset C(S^X_*)$.
For any function $f\in C(S^X_*)$ and  $\La\subset X$,
let $\iota^\La_* f\in C(S^\La)$ be the function
\[
	\iota^\La_* f(\e)\coloneqq f(\e|_\La), \quad\forall\e\in S^X_*,
\]
where $\e|_\La\in S^X_*$ is defined as the configuration whose component for $x\in \La$ coincides
with that of $\e$, and is at base state $*$ for components outside $\La$.
This defines a linear operator $\iota^\La_*\colon C(S^X_*)\rightarrow C(S^\La)\subset C(S^X_*)$.
In particular, if $\La$ is finite, then $\iota^\La_* f$ is a local function.
We define the space of \emph{local functions with exact support $\La$} by
\[	
	C_\La(S^X_*)\coloneqq\bigl\{f\in C(S^\La)\mid \text{$\forall$ finite $\La'\subset X$ such that 
	$\La\not\subset \La'$, we have $\iota^{\La'}_*f=0$}\bigr\}.
\]
A local function $f\in C(S^\La)$ is in $C_\La(S^X_*)$ if and only if $f(\e)=0$ for any 
configuration $\e=(\e_x)\in S^\La$ satisfying $\e_x=*$ for some $x\in\La$.
For any $\e=(\e_x)\in S^X$, we define the support of $\e$ by $\Supp(\e)\coloneqq\{x\in X\mid \e_x\neq*\}$.

\begin{lemma}\label{lem: sum}
	Suppose  $(f^*_{\La})$ is a system of functions such that $f^*_\La\in C_{\La}(S^X_*)$
	for any $\La\in\sI$.  Then $f=\sum_{\La\in\sI}f^*_\La$ gives a well-defined function on $S^X_*$.
\end{lemma}

\begin{proof}
	For any $\e\in S^X_*$,  we have $f^*_\La(\e)\neq0$ if and only if $\La\subset\Supp(\e)$.
	Since the support $\Supp(\e)$ is finite, 
	if we are given functions $f^*_\La\in C_\La(S^X_*)$ for any $\La\in\sI$, then
	we see that
	\[
		f(\e)=\sum_{\La\in\sI}f^*_{\La}(\e)=
		\sum_{\La\subset\Supp(\e)}f^*_{\La}(\e)
	\] 
	is a well-defined finite sum.  Hence the right hand side of \cref{eq: expansion} 
	defines a function on $S^X_*$.
\end{proof}

\begin{proposition}\label{prop: expansion}
	For any function $f\in C(S^X_*)$, there exists a unique system of functions  $(f^*_{\La})$ in $C_{\La}(S^X_*)$
	for any $\La\in\sI$ such that 
	\begin{equation}\label{eq: expansion}
		 f=\sum_{\La\in\sI}f^*_{\La}
	\end{equation}
	as a function on $S^X_*$.
\end{proposition}

\begin{proof}
	If such $(f^*_{\La})$ satisfying \eqref{eq: expansion} exists, 
	then by definition of 
	$C_{\La}(S^X_*)$, we see that 
	\begin{equation}\label{eq: induction}
	\iota^{\La}_* f=\sum_{\La''\subset\La}f^*_{\La''}
	\end{equation}
	for any $\La\in\sI$.
	We will construct $f^*_\La$ satisfying \eqref{eq: induction}
	by induction on the number of elements in $\La\in\sI$.
	If $\La=\emptyset$, then we let $f^*_\emptyset \coloneqq\iota^\emptyset_*f = f(\star)$.
	For any $\La\in\sI$, assume that there exists $f^*_{\La''}$ 
	for any $\La''\subsetneq\La$
	such that $\iota^{\La'}_* f=\sum_{\La''\subset\La'}f^*_{\La''}$  is satisfied for any proper subset $\La'\subset\La$.
	We let
	\[
		f^*_\La\coloneqq\iota_*^\La f - \sum_{\La''\subsetneq\La}f^*_{\La''}.
	\]
	Then for any $\La'\in\sI$ such that $\La\not\subset\La'$, we have $\La\cap\La'\subsetneq\La$.
	Hence
	\[
			\iota^{\La'}_*f^*_\La=\iota^{\La'}_*(\iota_*^{\La} f) - \sum_{\La''\subsetneq\La}\iota_*^{\La'}f^*_{\La''}
			=\iota_*^{\La\cap\La'} f - \sum_{\La''\subset\La\cap\La'}\iota_*^{\La'}f^*_{\La''}=0,
	\]
	where the
	second equality follows from the definition of $\iota^{\La'}_*$ and the fact that $f^*_{\La''}\in C_{\La''}(S^X)$ and
	the last equality follows from the induction hypothesis for $\La\cap\La'\subsetneq\La$.
	This shows that $f^*_\La\in C_\La(S^X)$, and that 
	$\iota^\La_*f=\sum_{\La''\subset\La}f^*_{\La''}$.  
	This gives a construction of a system $(f^*_\La)$ satisfying \eqref{eq: induction}.
	The uniqueness of $f^*_\La$ follows from the construction.
	By \cref{lem: sum}, the sum $\sum_{\La\in\sI}f^*_\La$ defines a well-defined function on $S^X_*$.
	\cref{eq: expansion} 
	follows from the fact that for any
	$\e\in S^X_*$, if $\Supp(\e)\subset\La$, then we have $f(\e)=\iota^\La_*f(\e)=
	\sum_{\La''\subset\La}f^*_{\La''}(\e)=
	\sum_{\La''\in\sI}f^*_{\La''}(\e)$ as desired.
\end{proof}

We may use \cref{prop: expansion} to define uniform functions in $C(S^X_*)$.
For any $x,y\in X$, we let $d_X(x,y)$ denote the graph distance of $x$ to $y$,
i.e. the length of the shortest path from $x$ to $y$, with the convention that $d_X(x,y)=0$
if $x=y$. 
For any $\La\subset X$, let $\diam(\La)\coloneqq\sup_{x,y\in\La}d_X(x,y)$.

\begin{definition}\label{def: uniform}
	We say that a function $f\in C(S^X_*)$ is \emph{uniform}, if 
	for the expansion of \cref{prop: expansion},
	there exists $R>0$ such that
	$f^*_\La=0$ if $\diam(\La)>R$.
\end{definition}

We denote by $C_\unif(S^X_*)$ the space of uniform functions.
A priori, the uniform functions would seem to depend on the choice
of the base state $*\in S$.
However, we may prove that there exists a canonical isomorphism between uniform 
functions for distinct base states $*$ and $*'$ as follows.
Let $\sI_R$ be the subset of $\sI$ consisting of finite $\La\subset X$ such that $\diam(\La)\leq R$.

\begin{proposition}\label{prop: independent}
	Let $*,*'\in S$ be any two states of $S$.
	Then we have an $\R$-linear isomorphism
	\[
		C_\unif(S^X_*)\cong C_\unif(S^X_\sd)
	\]
	which gives the map $f\mapsto f+(f(\star)-f(\star'))$ on the subspace of local functions $C_\loc(S^X)$.
\end{proposition}

\begin{proof}
	By \cref{prop: expansion} and the definition of uniform functions, 
	any $f\in C_\unif(S^X_*)$ has a decomposition
	\[
		f=\sum_{\La\in\sI_R}f^*_\La
	\]
	for some $R>0$,
	where $f^*_\La\in C_\La(S^X_*)\subset C(S^\La)$.   Since $C(S^\La)\subset C(S^X_\sd)$,
	we have a decomposition
	\[
		f^*_\La=\sum_{\La'\subset\La}(f^*_{\La})^\sd_{\La'},
	\]
	where $(f^*_{\La})^\sd_{\La'}\in C_{\La'}(S^X_\sd)$.  
	We define $f^\sd_{\La'}$ for $\La'\in\sI_R$ by
	$f^\sd_\emptyset\coloneqq f^*_\emptyset=f(\star)$
	for $\La'=\emptyset$ and
	for $\La'\neq\emptyset$, 
	we let
	\begin{equation}\label{eq:sd}
		f^\sd_{\La'}\coloneqq\sum_{\substack{\La\in \sI_R\\\La'\subset\La}}(f^*_{\La})^\sd_{\La'}.
	\end{equation}
	The sum is a finite sum since $\La'\neq\emptyset$ is finite and the sum is over $\La\in\sI_R$.
	Hence $f^\sd_{\La'}\in C_{\La'}(S^X_\sd)$ for any $\La'\in\sI$.	
	By \cref{lem: sum} and the definition of uniform functions, the infinite sum
	\begin{equation}\label{eq: A}
		u'(f)\coloneqq\sum_{\La'\in\sI_R}f^\sd_{\La'}
	\end{equation}
	defines a function in $C_\unif(S^X_\sd)$.  The correspondence $f\mapsto u'(f)$ gives 
	an $\R$-linear map
	\[
		u'\colon C_\unif(S^X_*)\rightarrow C_\unif(S^X_\sd)
	\]
	which satisfies $u'(f)(\star')=f(\star)$.
	By construction, we have $u'(f)=f+(f(\star)-f(\star'))$ on the space of local functions.

	We next define the inverse in a similar manner.  
	Suppose $*''\in S$ is another base state.
	If $f'$ is of the form $f'\coloneqq u'(f)=\sum_{\La'\in\sI_R}f^\sd_{\La'}$ 
	of \eqref{eq: A}, then for $\La'\in\sI_R$ and $\La''\in\sI_R$, we have
	\[
		(f^\sd_{\La'})^\sdd_{\La''}
		=
		\Bigl(\sum_{\substack{\La\in \sI_R\\\La'\subset\La}}(f^*_{\La})^\sd_{\La'}\Bigr)^\sdd_{\La''}
		=
		\sum_{\substack{\La\in \sI_R\\\La'\subset\La}}((f^*_{\La})^\sd_{\La'})^\sdd_{\La''}.
	\]
	We define $f^{*''}_{\La''}$ as in \eqref{eq:sd},
	so that $f^{*''}_\emptyset=f^{*'}_\emptyset=f^*_\emptyset=f(\star)$,
	and for $\La''\neq\emptyset$, we let
	\begin{align*}
		f^{*''}_{\La''}\coloneqq\sum_{\substack{\La'\in\sI_R\\\La''\subset\La'}}(f^\sd_{\La'})^\sdd_{\La''}
		&=
		\sum_{\substack{\La'\in\sI_R\\\La''\subset\La'}}
		\sum_{\substack{\La\in \sI_R\\\La'\subset\La}}((f^*_{\La})^\sd_{\La'})^\sdd_{\La''}
		=
		\sum_{\substack{\La\in \sI_R}}
		\sum_{\substack{\La'\in\sI_R\\\La''\subset\La'\subset\La}}
		((f^*_{\La})^\sd_{\La'})^\sdd_{\La''}
		=
		\sum_{\substack{\La\in \sI_R\\\La''\subset\La}}
		(f^*_{\La})^\sdd_{\La''}.
	\end{align*}
	Then by \eqref{eq: A}, we have
	\[
		u''(u'(f))\coloneqq\sum_{\La''\in\sI_R}f^{*''}_{\La''}
		=\sum_{\La''\in\sI_R}\sum_{\substack{\La\in\sI_R\\\La''\subset\La}}
		(f^*_{\La})^\sdd_{\La''},
	\]
	In particular, if $*''=*$, noting that $(f^*_{\La})^*_{\La''}=f^*_\La$ if $\La''=\La$ 
	and $(f^*_{\La})^*_{\La''}=0$ if $\La''\neq\La$, we have
	\[
		u''(u'(f))
		=\sum_{\La''\in\sI_R}\sum_{\substack{\La\in\sI_R\\\La''\subset\La}}
		(f^*_{\La})^*_{\La''}
		=\sum_{\La\in\sI_R}f^*_\La=f.
	\]
	This shows that $u''\circ u'=\id$ as desired.
	Reversing the roles of $*$ and $*'$, we may also prove that  $u'\circ u''=\id$,
	hence $u'$ is an isomorphism as desired.
\end{proof}

We will view uniform functions for different base points through the
canonical isomorphisms given by \cref{prop: independent}.
The isomorphism is given as a translation by constant functions on the 
space of local functions.  For this reason, it is convenient to consider the
space of functions modulo the subspace of constant functions.
Let $C^0_\loc(S^X)$ and $C^0_\unif(S^X_*)$ be the quotients of 
$C_\loc(S^X)$ and $C_\unif(S^X_*)$ modulo the subspace of constant functions.
Then \cref{prop: independent} gives the following.

\begin{corollary}\label{cor: independent}
	Let $*,*'\in S$ be any two states of $S$.
	Then the $\R$-linear isomorphism of \cref{prop: independent} induces an isomorphism
	\[
		C^0_\unif(S^X_*)\cong C^0_\unif(S^X_\sd)
	\]
	which is the identity on the subspace of local functions $C^0_\loc(S^X)$.
\end{corollary}

Due to \cref{cor: independent}, we will often simply denote the space $C^0_\unif(S^X_*)$
as $C^0_\unif(S^X)$ without specifying the base point $*\in S$.
For any $f\in C^0_\unif(S^X)$, we will call the function in $C^0_\unif(S^X_*)$ 
corresponding to $f$ a \emph{realization of $f$} for the base point $*\in S$.
Similarly,
for any $\La\subset X$,
we let $C^0(S^\La)$ be the quotient of 
$C(S^\La)$ modulo the subspace of constant functions.  For a base state $*\in S$,
any function $f\in C^0(S^\La)$ has a normalized representative $f\in C(S^\La)$ such that $f(\star)=0$.
This normalization depends on the choice of $*\in S$, but the corresponding class
in $C^0(S^\La)$ is independent of such choice.

Next, we show that the sum of a system $(f_x)_{x\in X}$ of uniformly local functions on $S^X$ 
defines a uniform function.
As stated in \S\ref{sec: introduction},
the premise of our model is that the observables of a configuration should \emph{not} depend on the 
entire configuration space, but should depend only on the local states in the proximity of the
point of observation.
Hence in order to express such observables, we introduce the notion of a function local at a vertex $x\in X$. 
Consider the symmetric digraph $(X,E)$.  
The graph distance $d_X$ of $(X,E)$ expresses the proximity of the vertices in $X$.
For any $R>0$ and $x\in X$, we let
\[
	B(x,R)\coloneqq\{ y\in X\mid d_X(x,y)<R\}
\]
be the \emph{ball} with center $x$ and radius $R$.  
We say that a function $f_x\colon S^X\rightarrow\R$ is \emph{local at 
$x$} (with radius $R$), if $f_x\in C(S^{B(x,R)})$ for some $R>0$.
In other words, 
$f_x(\e)$ for $\e=(\e_z)_{z\in X}$ depends only on the coordinates $(\e_z)_{z\in B(x,R)}$.

\begin{lemma}\label{lem: equivalent}
	A function $f\colon S^X\rightarrow\R$ is a 
	local function if and only if it is local at any $x\in X$.
\end{lemma}

\begin{proof}
	Suppose $f$ is a local function.  Then there exists finite $\La\subset X$ such that $f\in C(S^\La)$.
	Then, since $(X,E)$ is connected, for any $x\in X$,  there exists $R>0$ such that $\La\subset B(x,R)$.
	This shows that $f\in C(S^\La)\subset C(S^{B(x,R)})$, proving that $f$ is local at $x$.
	Conversely, suppose $f$ is local at $x\in X$.  Then there exists $R>0$ such that $f\in C(S^{B(x,R)})$.
	Since $(X,E)$ is locally finite, $B(x,R)$ is a finite set.
	This shows that $f$ is local as desired.
\end{proof}

Although equivalent by \cref{lem: equivalent} to the fact that a function is local,
the notion that a function is local at $x\in X$ allows for the definition of 
the uniformity of the localness as follows.

\begin{definition}
	Consider a system of functions $(f_x)_{x\in X}$ in $C^0(S^X)$.
	We say such system $(f_x)_{x\in X}$ is \emph{uniformly local}, if there exists
	$R>0$ such that $f_x$ is local at $x$ with radius $R$.
\end{definition}

Given a system $(f_x)_{x\in X}$ which is uniformly local,
the total observables of the entire system should be given as the
infinite sum $\sum_{x\in X}f_x$.
Note that $\e=(\e_x)\in S^X_*$ if and only if the support $\Supp(\e)$ is a finite set.
We have the following.

\begin{proposition}\label{prop: sum}
	Let $(f_x)_{x\in X}$ be a uniformly local system of functions in $C^0(S^X)$.
	Then the infinite sum
	\[
		f\coloneqq\sum_{x\in X}f_x
	\]
	defines a uniform function in $C^0_\unif(S^X)$.	
\end{proposition}

\begin{proof}
	Since $(f_x)_{x\in X}$  is uniformly local, there exists $R>0$ 
	such that $f_x\in C^0(S^{B(x,R)})$ for any $x\in X$.
	Take a base state $*\in S$, and a normalized representative 
	of $f_x\in C(S^{B(x,R)})$ satisfying $f_x(\star)=0$ for any $x\in X$.
	For $\e=(\e_x)\in S^X_*$,  by definition,
	the support $\Supp(\e)$ is a finite set.
	Since we have assumed that $(X,E)$ is locally finite, the vertices
	$x\in X$ such that $B(x,R)\cap\Supp(\e)\neq\emptyset$
	is also a finite set.  If $B(x,R)\cap\Supp(\e)=\emptyset$,
	then since $f_x\in C\bigl(S^{B(x,R)}\bigr)$, we have
	$f_x(\e)=f_x(\star)=0$.  This shows that $f(\e)=\sum_{x\in X}f_x(\e)$ is a 
	well-defined finite sum, hence defines a function in $C^0(S^X_*)$.
	Let $R>0$ such that $f_x\in C(S^{B(x,R)})$. If we take an expansion of $f_x$, then by \cref{prop: expansion},
	we have
	\[
		f_x=\iota^{B(x,R)}_*f_x=\sum_{\La\subset B(x,R)}(f_x)^*_\La.
	\]
	The uniqueness of the expansion of \cref{prop: expansion} 
	shows that for $f^*_\La$ such that $f=\sum_{\La\in\sI}f^*_\La$,
	we have
	\[
		f^*_\La=\sum_{x\in X, \\ \La\subset B(x,R)}(f_x)^*_\La
	\]
	where the sum is taken over all $x \in X$ such that $ \La\subset B(x,R)$. 
	This shows that $f^*_\La=0$ if $\diam(\La)>2R$, hence $f\in C^0_\unif(S^X)$ as desired.
\end{proof}

We may also prove the converse of \cref{prop: sum}.

\begin{proposition}\label{prop: converse}
	Let $f\in C^0_\unif(S^X)$.  Then there exists a system $(f_x)_{x\in X}$ of uniformly local functions in
	$C^0(S^X)$ such that $f=\sum_{x\in X}f_x$.
\end{proposition}

\begin{proof}
	We take a base state $*\in S$, and take $f\in C^0_\unif(S^X_*)$ normalized so that $f(\star)=0$.
	Then by the definition of uniform functions, there exists $R>0$ such that for the expansion \cref{eq: expansion},  
	$f^*_\La=0$ if $\diam(\La)>R$.  Moreover, $f^*_\emptyset=f(\star)=0$. For each $x \in X$, let 
\[
f_x \coloneqq \sum_{\La\in\sI,\,x \in \La } \frac{1}{|\La|}f^*_\La,
\]
	which is a well-defined finite sum such that $f_x\in C(S^{B(x,R)})$ since $\La\subset B(x,R)$ 
	if $\diam(\La)\leq R$.  We have
	\[
		\sum_{x\in X}f_x=\sum_{x\in X}\sum_{\La\in\sI,\,x \in \La } \frac{1}{|\La|}f^*_\La=\sum_{\La\in\sI}f^*_\La=f
	\]
	as a function on $S^X_*$, which shows that $f=\sum_{x\in X}f_x$ in $C^0_\unif(S^X)$ as desired.
\end{proof}

%%%%%%%%%%%%%%%%%%%%%%%%%%%%%%%%%%%%%%%%%%%%%%%%%%
%
%
%
\section{Configuration Space with Transition Structure}\label{sec: CS}
%
%
%
%%%%%%%%%%%%%%%%%%%%%%%%%%%%%%%%%%%%%%%%%%%%%%%%%%

In this section, we first define the notion of interactions and conserved quantities.  
Then we will consider the configuration space on an underlying graph given by the interaction,
and prove our main theorem.

\begin{definition}\label{def: interaction}
	Let $S$ be a non-empty set.
	We say that $\phi\subset(S\times S)\times (S\times S)$ is an \emph{interaction},
	if $(S\times S,\phi)$ is a symmetric digraph.  
	In other words, we have $(s_1,s_2)\in \phi$ if and only if $(s_2,s_1)\in\phi$. We also refer to the pair $(S,\phi)$ as an interaction,
and the graph $(S\times S,\phi)$ the associated graph.
\end{definition}

The set $S\times S$ represents all of the possible configurations on two adjacent vertices of an underlying graph,
and the associated graph expresses all of the possible transitions
of $S\times S$.

\begin{remark}
	In \cite{BKS20}*{Definition 2.4} and \cite{BS21}*{\S0}, we defined an interaction as a map 
	\[
		\phi_+\colon S\times S\rightarrow S\times S
	\] satisfying
	$
		\phi^\iota_+(s_1,s_2)=(s_1,s_2)
	$
	for any $(s_1,s_2)\in S\times S$ such that $ \phi_+(s_1,s_2)\neq (s_1,s_2)$.
	Here, we let $\phi^\iota_+\coloneqq \iota\circ\phi_+\circ\iota$, where
	$\iota(s_1,s_2)\coloneqq(s_2,s_1)$ for any $(s_1,s_2)\in S\times S$.
	Denote again by the same symbol
	\[
		\phi_+\coloneqq\{((s_1,s_2), \phi_+(s_1,s_2))\mid (s_1,s_2)\in S\times S\}\subset
	(S\times S)\times(S\times S)
	\] 
	the graph of $\phi_+$, 
	and let
	$\phi^\iota_+\coloneqq\{((s'_1,s'_2),\phi^\iota_+(s'_1,s'_2))\mid (s'_1,s'_2)\in S\times S\}$.
	Then the condition on the map $\phi_+$ insures that
	$\phi\coloneqq\phi_+\cup\phi^\iota_+\subset (S\times S)\times (S\times S)$ 
	is a symmtric digraph, which gives an interaction
	in our sense of \cref{def: interaction}.
\end{remark}

Next, we define a conserved quantity for an interaction.

\begin{definition}
	We say that a function $\xi\colon S\rightarrow\R$ is a \emph{conserved quantity} for an interaction $(S,\phi)$, if the function
	$\tilde{\xi}\colon S\times S\rightarrow\R$ given by
	\begin{equation}\label{eq: CQ}
		 \tilde{\xi}(s_1,s_2)
		\coloneqq\xi(s_1)+\xi(s_2)\quad\forall(s_1,s_2)\in S\times S
	\end{equation}
	is constant on the connected components of the associated graph $(S\times S,\phi)$.
\end{definition}

We remark that the constant function is trivially a conserved quantity.
We denote by $\C(S)$ the space of conserved quantities for an interaction $(S,\phi)$, modulo the subspace of constant functions.

An example of an interaction is given by the
multi-species exclusion interaction, which gives a generalization of the exclusion interaction.

\begin{figure}[htbp]
	\begin{center}
		\includegraphics[width=1\linewidth]{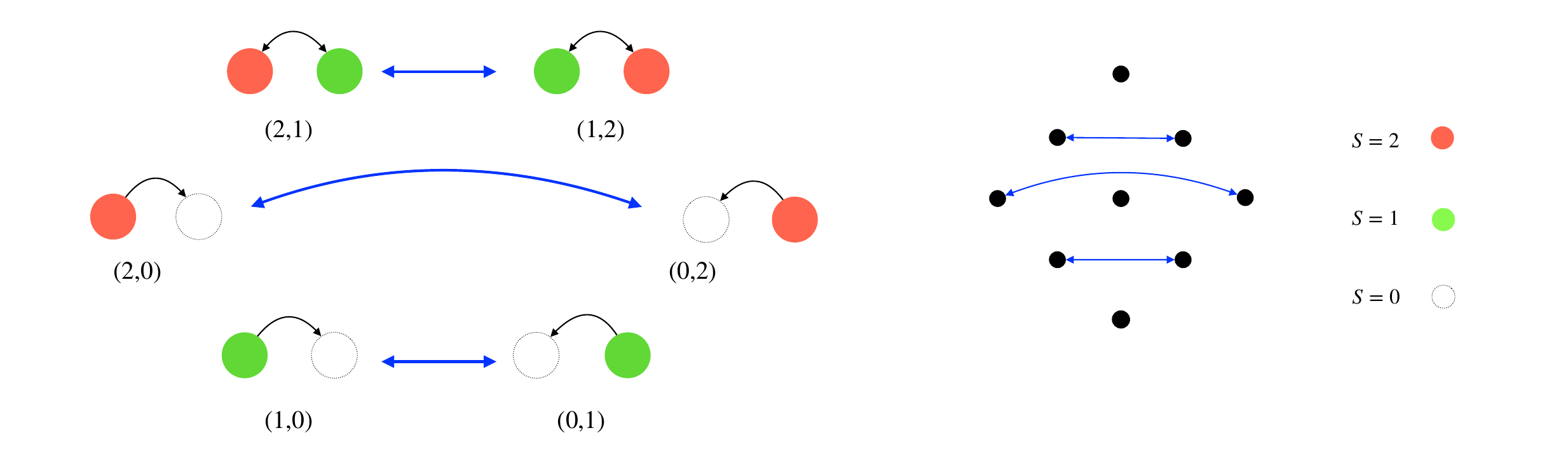}
		\caption{multi-species exclusion interaction for $\kappa=2$}
	\end{center}
\end{figure}

\begin{example}\label{example: 1}
	For an integer $\kappa\geq 1$ and $S_\kappa=\{0,1,\ldots,\kappa\}$,
	we let
	\[
		\phi^\kappa_\ms\coloneqq\{ ((j,k),(k,j))\in S_\kappa\times S_\kappa\mid j,k\in S_\kappa, j\neq k\}.
	\]
	Then the pair $(S_\kappa,\phi^\kappa_\ms)$ is an interaction, which we call the 
	\emph{multi-species exclusion interaction}.
	We have $c_{\phi^\kappa_\ms}=\dim_\R\Consv^{\phi^\kappa_\ms}(S_\kappa)=\kappa$, 
	and we have a basis  $\xi^1,\ldots,\xi^\kappa$ 
	of $\Consv^{\phi^\kappa_\ms}(S_\kappa)$ called the \emph{standard basis}
	defined as
	$\xi^i(0)=0$ and
	$\xi^i(j)=\delta_{ij}$ for any integer $i,j=1,\ldots,\kappa$.
	This interaction underlies the multi-color exclusion process studied by 
	Dermoune-Heinrich \cite{DH08} and Halim--Hac\`ene \cite{HH09}, as well as
	the multi-species exclusion process studied by Nagahata--Sasada \cite{NS11}.
	The case $\kappa=1$ coincides with the exclusion interaction, hence $\phi^1_\ms=\phi_\EX$.
\end{example}

The following two-species exclusion process with annihilation and creation
gives an example of an interaction
which does not come from a map as in \cite{BKS20}.
Compare \cite{BKS20}*{Example 2.10 (4)}.

\begin{figure}[htbp]
	\begin{center}
		\includegraphics[width=1\linewidth]{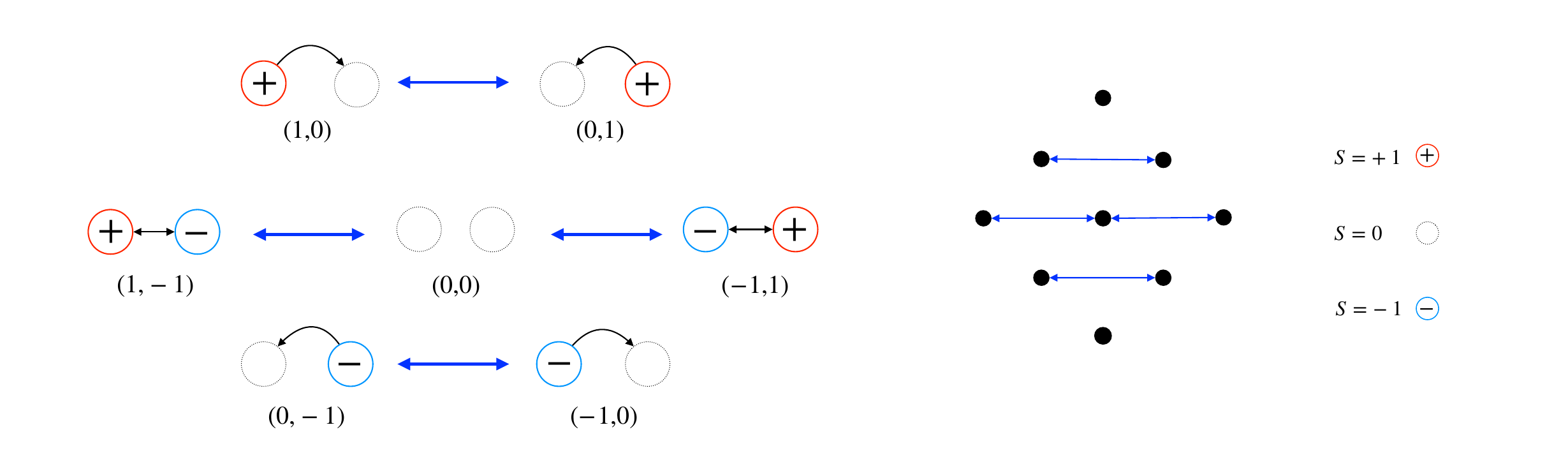}
		\caption{Interaction of \cref{example: 1} underlying the two-species exclusion 
		process with annihilation and creation studied in \cite{Sas10}.}
	\end{center}
\end{figure}

\begin{example}\label{example: 2}
	Let $S=\{-1,0,+1\}$, and we let $\phi\subset(S\times S)\times(S\times S)$ 
	be the interaction given by
	\begin{align*}
		(-1,0)&\leftrightarrow (0,-1), & (1,0)&\leftrightarrow (0,1), 
		& (1,-1)&\leftrightarrow (-1,1), &
		(1,-1)&\leftrightarrow(0,0)\leftrightarrow(-1,1),
	\end{align*}
	where $\leftrightarrow$ denotes the existence of an edge connecting the vertices.
	Then $(S,\phi)$ is an interaction underlying the two-species exclusion 
	process with annihilation and creation studied by Sasada in \cite{Sas10}.	
	We have $c_{\phi}=1$, and $\Consv^\phi(S)$ is spanned by 
	$\xi\colon S\rightarrow\R$ given by $\xi(j)=j$ for $j=-1,0,1$.
\end{example}

Note that the interactions of \cref{example: 1} and \cref{example: 2} are
both exchangeable in the sense of \cref{def: exchangeable}.
We remark that the two color exclusion process studied by Quastel \cite{Qua92} 
is a variant of \cref{example: 1} with $\kappa=2$ removing the edge $(1,2)\leftrightarrow(2,1)$,
hence is not exchangeable.

Now, let $(X,E)$ be a locally finite symmetric digraph.  For any conserved quantity $\xi\in\C(S)$,
if we let $\xi_x(\e)\coloneqq\xi(\e_x)$ for any $\e=(\e_x)\in S^X$, then $\xi_x$ gives a function
which is by definition local at $x$ of radius $0$.
Hence $(\xi_x)_{x\in X}$ is a system which is uniformly local.

\begin{definition}\label{def: CQX}
	For any conserved quantity $\xi\in\C(S)$, we let
	\[
		\xi_X\coloneqq\sum_{x\in X}\xi_x\in C^0_\unif(S^X)
	\]
	the uniform functions associated to the uniformly local system $(\xi_x)_{x\in X}$.
	The correspondence $\xi\mapsto\xi_X$ gives an $\R$-linear homomorphism $\C(S)\hookrightarrow C^0_\unif(S^X)$.
\end{definition}

We next consider the transition structure of an interaction on the configuration space. 
For any $e\in E$, we let $oe,te\in X$ be the origin and target of $e$ so that $e=(oe,te)$.
We define the transition structure $\Phi_E$ of an interaction $(S,\phi)$
on the configuration space $S^X$ as follows.

\begin{definition}\label{def: CSTS}
	We define the \emph{transition structure} $\Phi_E\subset S^X\times S^X$ 
	on the configuration space $S^X$ by
	\[
		\Phi_E\coloneqq\bigl\{(\e,\e')\in S^X\times S^X
		\mid \exists e\in E, \,\,  ((\e_{oe},\e_{te}),(\e'_{oe},\e'_{te}))\in\phi, \,\,\e_x=\e'_x\,\forall x \neq oe,te \bigr\}.
	\]
	Then since $(S\times S,\phi)$ is a symmetric digraph, the configuration space with
	transition structure $(S^X,\Phi_E)$ is also a symmetric digraph.
\end{definition}

The topology of the graph $(S^X,\Phi_E)$ reflects information concerning the large scale interacting system.  
If $\e$ and $\e'$ are in the same connected component of $(S^X,\Phi_E)$,
then this implies that there exists a path from $\e$ to $\e'$ in $(S^X,\Phi_E)$.
In this case, we say that there exists a finite sequence of transitions from $\e$ to $\e'$.
Thus, intuitively, a microscopic stochastic process associated to the interaction $(S,\phi)$ on a connected component of $(S^X,\Phi_E)$ would be irreducible.
Since macroscopically, the microscopic stochastic movements of the local states should be negligible, 
the observables of our microscopic configuration space relevant for the
associated macroscopic dynamics obtained by taking the hydrodynamic limit
should be invariant via transitions of $(S^X,\Phi_E)$.
This is the reason we are interested in functions on $S^X$ which are invariant via transitions.

The uniform functions of \cref{def: uniform} is not in itself a function on $S^X$.
Hence we first define what is meant for a uniform function to be invariant
via transitions of $(S^X,\Phi_E)$.  
For any $f\in C^0_\unif(S^X)$, since $f$ is not generally a function on $S^X$,
the values $f(\e)$ and $f(\e')$ may not be defined for $\e,\e'\in S^X$.  However, 
given a finite sequence of transitions from $\e$ to $\e'$,
we may define a well-defined difference $f(\e')-f(\e)$ as follows.
Due to this property, uniform functions may be understood as a kind of potential
on the configuration space.

\begin{lemma}\label{lem: difference}
	For any uniform function $f\in C^0_\unif(S^X)$ and $\e,\e'\in S^X$ such that
	 $\Delta_{\e,\e'}\coloneqq\{ x\in X\mid \e_x\neq\e'_x\}$ is a finite set,
	then the difference
	\[
		f(\e')-f(\e)
	\]
	extending the difference of local functions gives a well-defined value.
	In particular, given any $\e$ and $\e'$ in the same connected component of $(S^X, \Phi_E)$,
	the difference $f(\e')-f(\e)$ gives a well-defined value.
\end{lemma}

\begin{proof}
	Consider a base state $*\in S$ and an expansion 
	$f=\sum_{\La\in\sI_R}f^*_\La$ of \cref{prop: expansion} for some $R>0$.
	Then for any $\La\in\sI_R$ such that $\La\cap\Delta_{\e,\e'}=\emptyset$,
	we have $f^*_\La(\e)=f^*_\La(\e')$.  Hence the sum of the difference
	\begin{equation}\label{eq: difference}
		f(\e')-f(\e)\coloneqq
		\sum_{\La\cap\Delta_{\e,\e'}\neq\emptyset}(f^*_\La(\e')-f^*_\La(\e))
	\end{equation}
	gives a well-defined finite sum and is independent of the choice of the base state $*\in S$.
	By definition, this difference extends the difference for local functions.
	The statement for a finite sequence of transitions
	 from $\e$ to $\e'$ follows from the fact that $\La_{\e,\e'}$ is finite
	in this case.
\end{proof}

\begin{definition}
	We say that a uniform function $f$ is \emph{invariant via transitions of $(S^X,\Phi_E)$},
	if for any transition $(\e,\e')\in\Phi_E$,
	the difference $f(\e')-f(\e)=0$.
\end{definition}

Let $f$ be a uniform function which is invariant  via transitions of $(S^X,\Phi_E)$.
Then by definition, the realization $f\in C^0_\unif(S^X_*)$ 
for any base point $*\in S^X$ is a function which is constant on the connected components of 
$(S^X_*,\Phi_E^*)$ where $\Phi_E^*:=\Phi_E \cap (S^X_* \times S^X_*)$.
The uniform functions $\xi_X$ obtained from a conserved quantity $\xi\in\C(S)$
gives an example of a uniform function which is invariant via transitions.

\begin{lemma}\label{lem: CQ}
	For any conserved quantity $\xi\in\C(S)$, the uniform function $\xi_X$ of \cref{def: CQX} is invariant via transitions
	of $(S^X,\Phi_E)$.
\end{lemma}

\begin{proof}
	We select a base state $*\in S$.
	For any $\xi\in\C(S)$,
	we take the normalization $\xi(*)=0$.
	Then $\xi_x\in C_{\{x\}}(S^X_*)$.
	Hence the sum $\xi_X=\sum_{x\in X}\xi_x$ is in fact the expansion \cref{eq: expansion}.
	Suppose $(\e,\e')\in\Phi_E$.
	By the definition of $\Phi_E$, 
	there exists $e=(x,y)\in E$ such that $\e'_z=\e_z$ for $z\neq x,y$
	and $((\e_{x},\e_{y}),(\e'_{x},\e'_{y}))\in\phi$.
	Hence by 
	the definition of a conserved quantity \eqref{eq: CQ}, we have 
	$\xi(\e_{x})+\xi(\e_{y})=\xi(\e'_{x})+\xi(\e'_{y})$.
	This shows that 
	\[
		\xi_X(\e')-\xi_X(\e)=0,
	\]
	hence that $\xi_X$ is invariant via transitions of $(S^X,\Phi_E)$.
\end{proof}

As in \cref{def: exchangeable} of the introduction, we say that an interaction $(S,\phi)$ is \emph{exchangeable}
if and only if 
for any $(s_1,s_2)\in S\times S$,
the configurations
$(s_1,s_2)$ and $(s_2,s_1)$ are in the same connected component of $(S\times S,\phi)$.
We may now state our main theorem.

\begin{theorem}\label{thm: main}
	Suppose $(S,\phi)$ is an interaction which is exchangeable, 
	and let $(X,E)$ be a connected and locally finite symmetric digraph such that $X$ is infinite.
	Let $f\in C^0_\unif(S^X)$ be a uniform function which is invariant via transitions of
	$(S^X,\Phi_E)$.  
	Then there exists a conserved quantity $\xi\colon S\rightarrow\R$ such that $f=\xi_X$.
\end{theorem}

Note that contrary to \cite{BKS20}, we do not need to assume that the
interaction $(S,\phi)$ is irreducibly quantified.
The global observables of the
entire large scale interacting system expressed by a uniform function
which is invariant via transitions may be regarded as a global
conserved quantity.
The above theorem implies that the global conserved quantity of the
entire system
is always given as the sum of a local conserved quantity $\xi$ associated to a local interaction.

We will prove \cref{thm: main} at the end of this section.
In what follows, we assume that $(S,\phi)$ is an interaction which is exchangeable
and that $(X,E)$ is connected and locally finite.
We first prove that the reshuffling of the components of a configuration $\e=(\e_x)\in S^X$
in a configuration space for a connected graph
preserves the connected components of $(S^X,\Phi_E)$.
Consider $\e=(\e_x)\in S^X$.
For any $x,y\in X$, if we let $\e^{x,y}$ be the configuration obtained from $\e$ by exchanging the
$x$ and $y$ components so that $\e^{x,y}_x\coloneqq \e_y$ and $\e^{x,y}_y\coloneqq \e_x$.
For any $e\in E$, since we have assumed that $(S,\phi)$ is exchangeable,
$(\e_{oe},\e_{te})$ and $(\e_{te},\e_{oe})$ are in the same connected component of $(S\times S,\phi)$.  
Hence the configuration $\e^{oe,te}$
and $\e$ are in the same connected component of $(S^X,\Phi_E)$.
More generally, for $x,y\in X$, since we have assumed that
$(X,E)$ is connected, there exists a path $\vec\gamma=(e^1,\ldots,e^N)$
from $x$ to $y$.  By successively exchanging the components of $\e$ with respect to $oe^i,te^i$
for the edges $e^i$ in $\vec\gamma$ from $i=1$ to $N$ and then in the opposite order from $i=N-1$ to $i=1$, we see that $\e^{x,y}$ and $\e$ are in the connected component 
of $(S^X,\Phi_E)$.  This gives the following.

\begin{lemma}\label{lem: bijection}
	For a finite $\La\subset X$, let $\sigma\colon \La\rightarrow \La$
	be a bijection of sets. Then for any configuration $\e=(\e_x)\in S^X$, the configuration $\e^\sigma$
	given by $\e^\sigma_x\coloneqq\e_{\sigma x}$
	for $x\in\La$ and $\e^\sigma_x=\e_x$ for $x\not\in\La$
	is in the same connected component of $(S^X,\Phi_E)$ as that of $\e$.
\end{lemma}

\begin{proof}
	This follows from the fact that any bijection $\sigma\colon\La\rightarrow\La$ is obtained by
	successively exchanging two elements of $\La$, and that $\e^{x,y}$ and $\e$
	are in the same connected component of $(S^X,\Phi_E)$
	for any $x,y\in X$.
\end{proof}

We next prove the following.

\begin{lemma}\label{lem: vertex}
	Consider $f\in C(S^X_*)$ for a base state $*\in S$.
	For any $x\in X$, let $f^*_{\{x\}}$ be the function 
	with exact support $\{x\}$ in the expansion \cref{eq: expansion}.  If $f$ is invariant via
	transitions of $(S^X,\Phi_E)$, then the functions 
	\[
		f^*_{\{x\}}\colon S\rightarrow\R
	\]
	are equal as functions on $S$ for any $x\in X$.
\end{lemma}

\begin{proof}
	We let $s\in S$.  For $x,x'\in X$, consider configurations $\e=(\e_x),\e'=(\e'_x)\in S^X_*$
	such that the component of $\e$ at $x$ and $\e'$ and $x'$ coincides with $s$,
	and all of the other components coincide with the base state.
	By \cref{lem: bijection}, the configurations $\e$ and $\e'$ are in the same connected component of $S^X_*
	\subset S^X$, there exists a sequence of transitions from $\e$ to $\e'$.
	Since $f$ is invariant via transitions, we have 
	\[
		f^*_{\{x\}}(s)=f^*_{\{x\}}(\e)=f(\e)-f(\star)=f(\e')-f(\star)=f^*_{\{x'\}}(\e')=f^*_{\{x'\}}(s')
	\]
	as desired.  For the first and last equality, we have used the identification of $C(S)$
	with $C(S^{\{x\}})$ and $C(S^{\{x'\}})\subset C(S^X_*)$.
\end{proof}

\begin{lemma}\label{lem: single}
	Assume that $X$ is infinite,
	and let $f\in C^0_\unif(S^X)$. If $f$ is invariant via transitions of $(S^X,\Phi_E)$,
	then for a base state $*\in S$ and $f$ normalized so that $f(\star)=0$, we have
	\[
		f=\sum_{x\in X}f^*_{\{x\}}
	\]
	for the expansion of \cref{prop: expansion}.
\end{lemma}

\begin{proof}
	Since $f$ is uniform, there exists $R>0$ such that $f^*_\La=0$ if $\diam(\La)>R$.
	It is sufficient to prove that for any finite $\La\in\sI$ such that $\abs{\La}\geq2$,
	we have $f^*_\La=0$.  We consider $\La\in\sI$, and assume that $f^*_{\La''}=0$
	for any $\La''$ such that $2 \le \abs{\La''} < |\La|$.  This condition is trivially
	true if $\abs{\La}=2$.  Take $\La'\subset X$ such that $\abs{\La'}=\abs{\La}$
	and $\diam(\La')>R$.  Such $\La'$ exists since $(X,E)$ is locally finite and infinite.
	By construction, $f^*_{\La'}=0$.  Consider any $\e\in S^\La$, and denote again by $\e$
	the configuration in $S^X_*$ whose components coincides with that of $\e$
	for $x\in\La$ and is at base state for $x\not\in\La$.
	We fix a bijection $\La\cong\La'$, $x\mapsto x'$ and let $\e'$ be
	the configuration in $S^X_*$ such that $\e'_{x'}\coloneqq\e_x$ for $x'\in\La'$
	and is at base state for $x'\not\in\La'$.  Since $\e'$ is obtained from $\e$
	by rearranging the components, we see from \cref{lem: bijection} that $\e$ and $\e'$ are in the 
	same connected component of $S^X_*$.  Hence there exists a finite sequence of transitions from $\e$ to $\e'$.
	Since $f$ is invariant via transitions of $(S^X,\Phi_E)$, we have $f(\e)=f(\e')$.  Then
	\begin{align*}
		f(\e)&=\iota^\La_*f(\e)=f^*_\La(\e)+\sum_{\La''\subsetneq\La}f^*_{\La''}(\e)
		=f^*_\La(\e)+\sum_{x\in\La}f^*_{\{x\}}(\e),\\
		f(\e')&=\iota^{\La'}_*f(\e')=f^*_{\La'}(\e')+\sum_{\La''\subsetneq\La'}f^*_{\La''}(\e')
		=\sum_{x\in\La'}f^*_{\{x\}}(\e')
	\end{align*}
	since $f^*_{\La'}=0$.
	By \cref{lem: vertex}, we have $\sum_{x\in\La}f^*_{\{x\}}(\e)=\sum_{x\in\La'}f^*_{\{x\}}(\e')$.
	This shows that for any $\e\in S^\La$, we have $f^*_\La(\e)=0$, hence $f^*_\La=0$. Induction on the number of elements of $\Lambda$ completes the proof.
\end{proof}

\begin{proof}[Proof of \cref{thm: main}]
	For $x\in X$, let $\xi\coloneqq f^*_{\{x\}}\colon S\rightarrow\R$ viewed as a function on $S$.
	By \cref{lem: vertex}, the function $\xi$ on $S$ is independent of the choice of $x$,
	and we have $\xi_x=f^*_{\{x\}}$ as a function on $S^X$, and by
	\cref{prop: expansion} and \cref{lem: single}, 
	we have $f=\sum_{x\in X}\xi_x$.  In order to show that $\xi$ is a conserved quantity,
	consider an edge $e=(x_1,x_2)\in E$.  Consider any $(s_1,s_2),(s'_1,s'_2)\in S\times S$
	such that $((s_1,s_2),(s'_1,s'_2))\in\phi$.
	If we let $\e, \e'\in S^X_*$ to be the configuration such that $(\e_{x_1},\e_{x_2})=(s_1,s_2)$,
	$(\e'_{x_1},\e'_{x_2})=(s'_1,s'_2)$, and $\e_x,\e'_x$ is at base state for $x\neq x_1,x_2$.
	Then by the construction, we have $(\e,\e')\in\Phi_E$.	
	This shows that
	\[
		\xi(s_1)+\xi(s_2)=\xi_{x_1}(\e)+\xi_{x_2}(\e)=f(\e)=f(\e')
		=\xi_{x_1}(\e')+\xi_{x_2}(\e')=\xi(s'_1)+\xi(s'_2).
	\]
	This proves that $\xi$ is a conserved quantity, and $f=\sum_{x\in X}\xi_x=\xi_X$ as desired.
\end{proof}

%%%%%%%%%%%%%%%%%%%%%%%%%%%%%%%%%%%%%%%%%%%%%%%%%%
%
%
%
\section{The $0$-th Uniform Cohomology of the Configuration Space}\label{sec: cohomology}
%
%
%
%%%%%%%%%%%%%%%%%%%%%%%%%%%%%%%%%%%%%%%%%%%%%%%%%%

In this section, we will interpret our main result \cref{thm: main} in terms of
the $0$-th cohomology of the configuration space with transition structure $(S^X, \Phi_E)$.
We first review the cohomology of general graphs
(see for example \cite{BKS20}*{Appendix A}).

\begin{definition}\label{def: coh graph}
	For any symmetric digraph $(X,E)$, we let
	\begin{align*}
		C(X) &\coloneqq\Map(X, \bbR),&
		C^1(X) &\coloneqq\Map^\alt(E, \bbR),
	\end{align*}
	where $\Map^{\alt}(E, \bbR)
	\coloneqq\{\omega\colon E\rightarrow \bbR 
	\mid \forall (x,y)\in E\,\,\omega(y,x)=-\omega(x,y) \}$.
	Furthermore, we define the differential
	\begin{equation}\label{eq: diff A}
		\partial\colon C(X)\rightarrow C^1(X),\qquad f\mapsto\partial f
	\end{equation}
	by  $\partial f(e)\coloneqq f(te) - f(oe)$ for any $e\in E$.	
	We define the \textit{cohomology} of $(X,E)$ by
	\begin{align*}
		H^0(X)&\coloneqq \Ker\partial, &  H^1(X)&\coloneqq C^1(X)/\partial C(X),
	\end{align*}
	and $H^m(X)\coloneqq\{0\}$
	for any $m\in\bbN$ such that $m\neq0,1$.
\end{definition}

The $0$-th cohomology $H^0(X)$ is known to be the space of 
functions which are constant on the connected components of $(X,E)$ (see \cite{BKS20}*{Proposition A.8}).
We will apply the above construction to the configuration space with transition structure.
Let $(S,\phi)$ be an interaction, and let $(X,E)$ be a locally finite connected symmetric digraph.
We let $(S^X,\Phi_E)$ be the configuration space with transition  structure of \cref{def: CSTS}
associated to $(S,\phi)$ and $(X,E)$.  Following \cref{def: coh graph}, the cohomology of the
configuration space with transition structure is given by
\begin{align*}
		C(S^X) &\coloneqq\Map(S^X, \bbR),&
		C^1(S^X) &\coloneqq\Map^\alt(\Phi_E, \bbR),
\end{align*}
where $\Map^{\alt}(\Phi_E, \bbR)\coloneqq\{\omega\colon \Phi_E\rightarrow \bbR 
\mid \forall(\e,\e')\in \Phi_E\,\,\omega(\e',\e)=-\omega(\e,\e') \}$.
Furthermore, we define the differential
\begin{equation}\label{eq: diff A}
	\partial\colon C(S^X)\rightarrow C^1(S^X),\qquad f\mapsto\partial f
\end{equation}
by  $\partial f(\e,\e')\coloneqq f(\e') - f(\e)$ for any $(\e,\e')\in\Phi_E$.	
We define the \textit{cohomology} of $(S^X,\Phi_E)$ by
\begin{align}\label{eq: H0}
	H^0(S^X)&\coloneqq \Ker\partial, &  H^1(S^X)&\coloneqq C^1(S^X)/\partial C(S^X),
\end{align}
and $H^m(S^X)\coloneqq\{0\}$
for any $m\in\bbN$ such that $m\neq0,1$. 

When $X$ is infinite, the graph $(S^X,\Phi_E)$ generally has an infinite number of connected
components.  Since the cohomology $H^0(S^X)$
is equivalent to functions on $X$ which are constant on the connected components
of $(X,E)$, this implies that $H^0(S^X)$ is infinite dimensional.
We will replace the functions $C(S^X)=\Map(S^X,\R)$ with the space of uniform functions
to construct a suitable definition of uniform cohomology.

First, since the constant function maps to zero, the differential \cref{eq: diff A}
induces a differential
$
	\partial\colon C^0(S^X)\rightarrow C^1(S^X).
$
The restriction to local functions gives 
\begin{equation}\label{eq: loc-diff}
	\partial\colon C^0_\loc(S^X)\rightarrow C^1(S^X).
\end{equation}
We may linearly extend this differential to uniform functions as follows.

\begin{proposition}
	The differential $\partial\colon C^0_\loc(S^X)\rightarrow C^1(S^X)$ extends to a differential
	\[
		\partial\colon C^0_\unif(S^X)\rightarrow C^1(S^X)
	\]
	on the space of uniform functions.
\end{proposition}

\begin{proof}
	For any $(\e,\e')\in\Phi_E$, by definition, there exists $(x,y)\in E$ such that $\e_z=\e'_z$ for $z\neq x,y$.
	Hence $\Delta_{\e,\e'}=\{ x \in X\mid \e_x \neq\e'_x \}$ is a finite set.
	By \cref{lem: difference}, we have a well-defined difference $\p f(\e,\e')=f(\e')-f(\e)$ .
	Hence defines a function $\partial f\colon\Phi_E\rightarrow\R$.
	The map $f\mapsto\partial f$ gives a differential
	linearly extending the differential \cref{eq: loc-diff}.
\end{proof}

Using this differential, we may define the $0$-th uniform cohomology as follows.

\begin{definition}
	Following \cref{eq: H0}, we define the $0$-th uniform cohomology of $(S^X,\Phi_E)$ by
	\[
		H^0_\unif(S^X)\coloneqq\Ker\bigl(\partial\colon C^0_\unif(S^X)\rightarrow C^1(S^X)\bigr).
	\]
\end{definition}

As an analogy of the fact that $0$-th cohomology $H^0(X)$ coincides with the space of functions
constant on the connected components of $(X,E)$ (see \cite{BKS20}*{Proposition A.8}),
we may prove that the $0$-th uniform cohomology $H^0_\unif(S^X)$
coincides with the space of functions constant on the connected 
components of $(S^X,\Phi_E)$, i.e.\,
functions invariant via transitions of $(S^X,\Phi_E)$.

\begin{lemma}\label{lem: connected}
	Consider a uniform function $f\in C^0_\unif(S^X)$.
	Then  $f\in H^0_\unif(S^X)$ if and only if $f$ is invariant via transitions of $(S^X,\Phi_E)$.
\end{lemma}

\begin{proof}
	Let $f\in H^0_\unif(S^X)$.  
	Then for any $(\e,\e')\in \Phi_E$,
	since $\p f=0$, we have
	\[
		\p f(\e,\e')=f(\e')-f(\e)=0.
	\]
	This proves that $f$ is invariant via transitions of $(S^X,\Phi_E)$.
	Conversely, suppose $f$ is invariant via transitions of $(S^X,\Phi_E)$.
	Then for any $(\e,\e')\in\Phi_E$, we have $\p f(\e,\e')=f(\e')-f(\e)=0$.
	This proves that $f\in H^0_\unif(S^X)$ as desired.
\end{proof}

Our main theorem gives the following corollary.

\begin{corollary}\label{cor: main}
	Let $(S,\phi)$ be an interaction which is exchangeable, and let $(X,E)$ be a connected
	and locally finite symmetric digraph
	such that $X$ is infinite.
	Then we have a canonical isomorphism $\C(S)\cong H^0_\unif(S^X)$, given by $\xi\mapsto\xi_X$
	for any conserved quantity $\xi\in\C(S)$.
\end{corollary}

\begin{proof}
	Suppose $\xi\in\C(S)$.  Then by \cref{lem: CQ}, the uniform function
	$\xi_X$ is invariant via transitions of $(S^X,\Phi_E)$.
	Hence by \cref{lem: connected}, we have $\xi_X\in H^0_\unif(S^X)$.
	This gives a homomorphism $\iota\colon\C(S)\rightarrow H^0_\unif(S^X)$.
	Fix a base $*\in S$ and take a representative of $\xi$ so that $\xi(*)=0$,
	and let $x\in X$.  Then for any $s\in S$, let $\e^s\in S^X$ be the configuration
	such that $\e^s_x\coloneqq s$ and $\e^s_z\coloneqq*$ for $z\neq x$.  Then $\xi(s)=\xi_X(\e^s)$.
	Hence if $\xi_X=0$ in $H^0_\unif(S^X)$, then this implies that $\xi(s)=0$
	for any $s\in S$, hence $\iota$ is injective.
	Next, suppose $f\in H^0_\unif(S^X)$. Then by
	\cref{lem: connected}, the uniform function $f$ is invariant via 
	transitions of $(S^X,\Phi_E)$.  Hence by \cref{thm: main}, 
	there exists $\xi\in\C(S)$ such that $f=\xi_X$.
	This proves that $\iota$ is surjective, hence that we have an 
	isomorphism $\C(S)\cong H^0_\unif(S^X)$ as desired.
\end{proof}

\subsection*{Acknowledgement}
The authors would like to thank members of the Hydrodynamic Limit Seminar
at Keio/RIKEN, especially Fuyuta Komura, Jun Koriki, 
Hidetada Wachi, Hayate Suda and Hiroko Sekisaka 
for discussion and their continual support for this project.

%%%%%%%%%%%%%%%%%%%%%%%%%%%%%%%%%%%%
% Reference

%%%%%%%%%%%%%%%%%%%%%%%%%%%%%%%%%%%%%%%%%%%%%%%%%%
\end{document}